\documentclass[11pt]{amsart}
\usepackage{amssymb,latexsym,amsmath,amsfonts}
\usepackage[mathscr]{eucal}
\numberwithin{equation}{section}
\theoremstyle{definition}
\newtheorem{definition}{Definition}[section]
\newtheorem{example}[definition]{Example}
\newtheorem{custom}[definition]{}
\theoremstyle{remark}
\newtheorem{remark}[definition]{Remark}
\theoremstyle{plain}
\newtheorem{theorem}[definition]{Theorem}
\newtheorem{lemma}[definition]{Lemma}

\newtheorem{result}[definition]{Result}
\newtheorem{obs}[definition]{Observation}

\newcommand{\eps}{\varepsilon}

\newcommand{\zt}{\zeta}

\newcommand{\lam}{\lambda}
\newcommand{\exep}{\mathfrak{S}}

\newcommand{\banal}{\mathcal{A}}
\newcommand{\id}{\mathbb{I}}
\newcommand{\f}{\boldsymbol{\sf f}}

\newcommand{\bdy}{\partial\mathbb{D}}

\newcommand{\OM}{\Omega}
\newcommand{\Dee}{\mathbb{D}}
\newcommand{\cDee}{\overline{\mathbb{D}}}

\newcommand{\smoo}{\mathcal{C}}
\newcommand{\hol}{\mathcal{O}}

\newcommand{\poinc}{p_{\mathbb{D}}}
\newcommand\Lem[1]{\widetilde{\kappa}_{#1}}

\newcommand\hyper[2]{\left|\frac{{#1}-{#2}}{1-\overline{{#2}}{#1}}\right|}
\newcommand\ilnhyper[2]{\left|({#1}-{#2})(1-\overline{{#2}}{#1})^{-1}\right|}
\newcommand{\symdist}{\mathfrak{p}_n}

\newcommand{\bcdot}{\boldsymbol{\cdot}}

\newcommand\blah[2]{\left(\frac{{#1}-{#2}}{1-\overline{{#2}}{#1}}\right)}
\newcommand{\mobi}{\mathcal{M}_{\mathbb{D}}}
\newcommand{\mobd}{{\sf dist}_{\mathcal{M}}}
\newcommand{\impl}{\Longrightarrow}
\newcommand{\mapp}{\longrightarrow}

\newcommand{\Fd}{F_d}

\newcommand{\shrp}{{\sf G}^{A,d}}
\newcommand{\Shrp}{{\sf G}^{(d)}}

\newcommand{\ess}{\widetilde{{}s}}

\newcommand{\tr}{{\sf tr}}
\newcommand{\lammu}{\lambda^{(\mu)}}

\newcommand{\Cn}{\mathbb{C}^n}
\newcommand{\cplx}{\mathbb{C}}

\begin{document}

\title[Interpolation in the spectral unit ball]{Some new observations on interpolation \\ 
in the spectral unit ball}
\author{Gautam Bharali}
\address{Department of Mathematics, Indian Institute of Science, Bangalore -- 560 012}
\email{bharali@math.iisc.ernet.in}
\thanks{This work is supported in part by a grant from the UGC under DSA-SAP, Phase~IV. \\
To appear in {\bf{\em Integral Eqns. Operator Theory}}}

\begin{abstract} 
We present several results associated to a holomorphic-interpolation problem for the spectral
unit ball $\OM_n, \ n\geq 2$. We begin by showing that a known necessary condition for the existence
of a $\hol(\Dee;\OM_n)$-interpolant ($\Dee$ here being the unit disc in $\cplx$), given that the
matricial data are non-derogatory, is not sufficient. We provide next a new necessary condition
for the solvability of the two-point interpolation problem -- one which is {\em not} restricted only
to non-derogatory data, and which incorporates the Jordan structure of the prescribed data. We then
use some of the ideas used in deducing the latter result to prove a Schwarz-type lemma for 
holomorphic self-maps of $\OM_n, \ n\geq 2$.
\end{abstract}
\maketitle

\section{Introduction and Statement of Results}\label{S:intro}

The interpolation problem referred to in the title, and which links the assorted results
of this paper, is the following ($\Dee$ here will denote the open unit disc centered at $0\in\cplx$):

\begin{itemize}
\item[(*)] {\em Given $M$ distinct points
$\zt_1,\dots,\zt_M\in \Dee$ and matrices
$W_1,\dots, W_M$ in the spectral unit ball $\OM_n:=\{W\in M_n(\cplx):r(W)<1\}$,
find conditions on $\{\zt_1,\dots,\zt_M\}$ and $\{W_1,\dots,W_M\}$ such that there
exists a holomorphic map $F:\Dee\longrightarrow\OM_n$ satisfying $F(\zt_j)=W_j, \
j=1,\dots,M$.}
\end{itemize}

\noindent{In the above statement, $r(W)$ denotes the spectral radius of the $n\times n$
matrix $W$. Under a very slight simplification  -- i.e. that the interpolant $F$ in (*)
is required to satisfy $\sup_{\zt\in\Dee}r(F(\zt))<1$ -- the paper 
\cite{bercFoiasTann:svNPip90} provides a characterisation of the interpolation data
$((\zt_1,W_1),\dots,(\zt_M,W_M))$ that admit an interpolant of the type described. However,
this characterisation involves a non-trivial search over a region in $\cplx^{n^2M}$. Thus,
there is interest in finding alternative characterisations that either: 
{\em a)} circumvent the need to perform a search; or {\em b)} reduce the dimension of the
search-region. In this regard, a new idea idea was introduced by Agler~\&~Young in the
paper \cite{aglerYoung:cldC2si99}. This idea was further developed over several 
works -- notably in \cite{aglerYoung:2psNPp00}, in the papers \cite{costara:22sNPp05} 
and \cite{costara:osNPp05} by Costara, and in David Ogle's thesis \cite{ogle:oftsp99}. It
can be summarised in two steps as follows:
\begin{itemize}
\item If the matrices $W_1,\dots,W_M$ are all non-derogatory, then (*) is equivalent to
an interpolation problem in the {\em symmetrized polydisc} $G_n, \ n\geq 2$, which is
defined as
\[
G_n := \left\{(s_1,\dots,s_n)\in\cplx^n: \text{all the roots of} \  
z^n+\sum\nolimits_{j=1}^n(-1)^js_jz^{n-j}=0 \ \text{lie in $\Dee$}\right\}.
\]
\smallskip
  
\item The $G_n$-interpolation problem is shown to share certain aspects of the classical
Nevanlinna-Pick problems, either by establishing conditions for a von Neumann inequality 
for $\overline{G}_n$ -- note that $\overline{G}_n$ is compact -- or through function theory.
\end{itemize}

\noindent{It would be useful, at this stage, to recall the following}

\begin{definition}\label{D:nonderog} A matrix $A\in M_n(\cplx)$ is said to be {\em non-derogatory}
if the geometric multiplicity of each eigenvalue of $A$ is $1$ (regardless of its algebraic
multiplicity). The matrix $A$ being non-derogatory is equivalent to $A$ being similar
to the companion matrix of its characteristic polynomial -- i.e.,
if $z^n+\sum_{j=1}^ns_jz^{n-j}$ is the characteristic polynomial then
\[
A \ \text{\em is non-derogatory} \ \iff \ \text{\em $A$ is similar to}
                 \begin{bmatrix}
                        \ 0  & {} & {} & -s_n \ \\
                        \ 1  & 0  & {} & -s_{n-1} \ \\
                        \ {} & \ddots  & \ddots & \vdots \ \\
                                \ \text{\LARGE{0}} &   & 1 & -s_{1} \
                \end{bmatrix}_{n\times n}.
\]
\end{definition}
\smallskip

The Agler-Young papers treat the case $n=2$, while the last two works cited above consider
the higher-dimensional problem. The reader is referred to \cite{aglerYoung:2psNPp00} for
a proof of the equivalence of (*), given non-derogatory matricial data, and the appropriate
$G_n$-interpolation problem. The similarity condition given in Definition~\ref{D:nonderog}
is central to establishing this equivalence.
\smallskip

Before presenting the first result of this paper, we need to examine what is known
about (*) from the perspective of the $G_n$-interpolation problem. Since we would like to
focus on the matricial interpolation problem, we will paraphrase the results from 
\cite{ogle:oftsp99} and \cite{costara:osNPp05} in the language of non-derogatory matrices.
Given an $n\times n$ complex matrix $W$, let its characteristic polynomial
$\chi^W(z)=z^n+\sum_{j=1}^n(-1)^js_j(W)z^{n-j}$, and define the rational
function
\[
\f(z;W) \ := \ \frac{\sum_{j=1}^njs_j(W)(-1)^jz^{j-1}}
				{\sum_{j=0}^{n-1}(n-j)s_j(W)(-1)^jz^{j}}.
\]
Then, the most general statement that is known about (*) is:

\begin{result}[{paraphrased from \cite{ogle:oftsp99} and \cite{costara:osNPp05}}]
\label{R:neccCon} Let $\zt_1,\dots,\zt_M$ be $M$ distinct points in $\Dee$ and
let $W_1,\dots, W_M\in\OM_n$ be non-derogatory matrices. If there exists a map
$F\in\hol(\Dee,\OM_n)$ such that $F(\zt_j)=W_j, \ j=1,\dots,M$, then the matrices
\begin{equation}\label{E:necc}
\left[\frac{1-\overline{\f(z};W_j)\f(z;W_k)}{1-\overline{\zt_j}\zt_k}\right]_{j,k=1}^M \ 
	\geq \ 0 \quad\text{for each $z\in\overline{\Dee}$}.
\end{equation}
\end{result}
\smallskip

\noindent{Here, and elsewhere in this paper, given two complex domains $X$ and $Y$, 
$\hol(X;Y)$ will denote the class of all holomorphic maps from $X$ into $Y$.}
\smallskip

\begin{remark} The matrices in \eqref{E:necc} may appear different from those in 
\cite[Corollary 5.2.2]{ogle:oftsp99}, but the latter are, in fact, 
$*$-congruent to the matrices above.
\end{remark}
\smallskip

Even though Result~\ref{R:neccCon} provides only a necessary condition, \eqref{E:necc} is more 
tractable for small values of $M$ than the Bercovici-Foias-Tannenbaum condition. Its viability as 
a sufficient condition, at least for 
small $M$, has been discussed in both \cite{ogle:oftsp99} and \cite{costara:osNPp05}.
This is reasonable because {\em the latter condition is sufficient} when $n=2$ and $M=2$ 
(and the given matrices are, of course, non-derogatory); see \cite{aglerYoung:hgsb04}. Given
all these developments, it seems appropriate to begin with the following:

\begin{obs}\label{O:wontDo} When $n\geq 3$, the condition \eqref{E:necc} is {\em not sufficient}
for the existence of a $\hol(\Dee;\OM_n)$-interpolant for the prescribed data 
$((\zt_1,W_1),\dots,(\zt_M,W_M))$, where each $W_j\in\OM_n, \ j=1,\dots,M$, is non-derogatory.
\end{obs}
\smallskip

The above observation relies on ideas from complex geometry; specifically -- estimates for 
invariant metrics on the symmetrized polydisc $G_n, \ n\geq 3$. Our argument follows from a
recent study \cite{nikPflThoZwo:ecmsp07} of the Carath{\'e}odory metric on $G_n, \ n\geq 3$. This
argument is presented in the next section.
\smallskip

Observation~\ref{O:wontDo} takes us back to the drawing board when it comes to realising
goals of the type {\em (a)} or {\em (b)} (as in the opening paragraph) to determine
whether a $\hol(\Dee;\OM_n)$-interpolant exists for a given data-set. Thus, new conditions 
that are inequivalent to \eqref{E:necc} are desirable for the same reasons as those offered in
\cite{aglerYoung:2psNPp00} and \cite{aglerYoung:2b2sNPp04}. To wit: all extant
approaches to implementing the Bercovici-Foias-Tannenbaum solution of (*) are computational,
and rely upon various search algorithms. Rigorous analytical results, even if they only
indicate when a data-set $((\zt_1,W_1),\dots,(\zt_M,W_M))$ does {\em not} admit an
$\hol(\Dee;\OM_n)$ interpolant -- i.e. necessary conditions -- provide tests of existing
algorithms/software and illustrate the complexities of (*). We will say more about
this; but first -- notations for our next result. Given $z_1,z_2\in\Dee$, the 
{\em pseudohyperbolic distance} between these points, written $\mobi(z_1,z_2)$, is defined as:
\[
\mobi(z_1,z_2) \ := \ \hyper{z_1}{z_2} \quad\forall z_1,z_2\in\Dee
\]
We can now state our next result.

\begin{theorem}\label{T:disc} Let $F\in\hol(\Dee;\OM_n)$, $n\geq 2$, and let $\zt_1,\zt_2\in\Dee$. 
Write $W_j=F(\zt_j)$, and let 
\begin{align}
\sigma(W_j) \ &:= \ \text{the {\em set} of eigenvalues of $W_j, \ j=1,2$} \notag \\ 
	&\quad \ \text{(i.e.~elements of $\sigma(W_j)$ are {\em not}
				 repeated according to multiplicity).}\notag 
\end{align}
If $\lam\in\sigma(W_j)$, then let $m(\lam)$ denote the multiplicity of $\lam$ as a zero of the
minimal polynomial of $W_j$. Then:
\begin{equation}\label{E:SchwarzIneq}
\max\left\{\max_{\mu\in\sigma(W_2)}\prod_{\lam\in\sigma(W_1)}\mobi(\mu,\lam)^{m(\lam)}, 
\ \max_{\lambda\in\sigma(W_1)}\prod_{\mu\in\sigma(W_2)}\mobi(\lam,\mu)^{m(\mu)}\right\} \ 
\leq \ \hyper{\zt_1}{\zt_2}.
\end{equation}
\end{theorem}
\smallskip

Referring back to our previous paragraph: one could ask whether Theorem~\ref{T:disc} is able to
highlight any complexities of (*) that Result~\ref{R:neccCon} misses. There are two parts 
to the answer:

\begin{itemize}
\item[1)] {\em The Jordan structure of the data-set $((\zt_1,W_1),(\zt_2,W_2))$:} Several 
well-known examples from \cite{bercFoiasTann:sclt91} and \cite{aglerYoung:2psNPp00} reveal that
the existence of a $\hol(\Dee;\OM_n)$-interpolant, $n\geq 2$, is sensitive to the 
Jordan structure of the matrices $W_1,\dots,W_M$. However, to the best of our knowledge, there
are no results in the literature to date that incorporate information on the Jordan structures or 
the minimal polynomials of $W_1,\dots,W_M$. In contrast, the following example shows that
information on minimal polynomials is vital -- i.e. that with the correct information
about the minimal polynomials of $F(\zt_1)$ and $F(\zt_2)$, condition \eqref{E:SchwarzIneq}
is sharp.

\begin{example}\label{Ex:keyex}
For $n\geq 3$ and $d=2,\dots,n-1$, define the holomorphic map $\Fd:\Dee\mapp\OM_n$ by
\[
\Fd(\zt) \ := \ \begin{bmatrix}
                \ 0  & {} & {} & \zt \ &\vline & \ {} \ \\
                \ 1  & 0  & {} & 0 \ &\vline & {} \ \\
                \ {} & \ddots  & \ddots & \vdots \ &\vline & \text{\LARGE{0}} \ \\
                \ {} & {} & 1 & 0 \ &\vline & {} \ \\ \hline
                \ {} & {} & {} & {} \ &\vline & {} \ \\
                \ {} & {} & \text{\LARGE{0}} & {} \ &\vline & \zt\mathbb{I}_{n-d} \
                \end{bmatrix}_{n\times n}, \qquad \zt\in \Dee,
\]
where $\mathbb{I}_{n-d}$ denotes the identity matrix of dimension $n-d$ for $1<d<n$.
Let $\zt_1=0$ and $\zt_2=\zt$. One easily computes -- in the notation of 
Theorem~\ref{T:disc} -- that:
\begin{align} 
\max_{\mu\in\sigma(W_2)}\prod\nolimits_{\lam\in\sigma(W_1)}\mobi(\mu,\lam)^{m(\lam)} \ &= \ |\zt|,
\notag \\
\max_{\lambda\in\sigma(W_1)}\prod\nolimits_{\mu\in\sigma(W_2)}\mobi(\lam,\mu)^{m(\mu)} \ &= \ |\zt|^2,
\notag
\end{align}
where the first equality holds because $W_1$ is nilpotent of order $d$.
So, \eqref{E:SchwarzIneq} is satisfied as an equality for the given choice of $\zt_1$ and 
$\zt_2$ -- which is what was meant above by saying that \eqref{E:SchwarzIneq} {\em is sharp}.\qed
\end{example}

\item[2)] {\em Comparison with \eqref{E:necc}:} Theorem~\ref{T:disc} would not be effective
in testing any of the existing algorithms used in the implementation of the 
Bercovici-Foias-Tannenbaum solution to (*) if \eqref{E:necc} were a universally stronger
necessary condition than \eqref{E:SchwarzIneq}. However, \eqref{E:necc} is devised 
with non-derogatory data in mind, whereas no {\em simple} interpolation condition was hitherto 
known for pairs of arbitrary matrices in $\OM_n$. Hence, by choosing any one of $W_1$ and
$W_2$ to be derogatory, one would like to examine how \eqref{E:necc} and \eqref{E:SchwarzIneq}
compare. This leads to our next observation.
\end{itemize}

\begin{obs}\label{O:compare} For each $n\geq 3$, we can find a data-set 
$((\zt_1,W_1),(\zt_2,W_2))$ for which \eqref{E:SchwarzIneq} implies that it cannot admit 
any $\hol(\Dee;\OM_n)$-interpolant, whereas \eqref{E:necc} provides no information.
\end{obs}
\smallskip

An example pertinent to this observation is presented at the end of 
Section~\ref{S:proofDisc}. As for Theorem~\ref{T:disc}, it may be viewed as a 
Schwarz lemma for mappings between $\Dee$ and the spectral unit ball. Note that the
inequality \eqref{E:SchwarzIneq} is preserved under automorphisms of $\Dee$ and under
the ``obvious'' automorphisms of $\OM_n$ (the full automorphism group 
${\sf Aut}(\OM_n), \ n\geq 2$, is not known). The proof of Theorem~\ref{T:disc} is
presented in Section~\ref{S:proofDisc}.
\smallskip

The key new idea in the proof of Theorem~\ref{T:disc} -- i.e.~to focus on the minimal
polynomial of certain crucial matrices that lie in the range of $F$ -- pays off in 
obtaining a result that is somewhat removed from the our main theme. The result in question
is a generalisation of the following theorem of Ransford and White
\cite[Theorem~2]{ransfordWhite:hsmsub91}:
\begin{equation}\label{E:rnsfrdWht}
G\in\hol(\OM_n;\OM_n) \ \text{and} \ G(0)=0 \ \impl \ r(G(X))\leq r(X) \;\; \forall X\in\OM_n.
\end{equation}
One would like to generalise \eqref{E:rnsfrdWht} in the way the Schwarz-Pick lemma generalises
the Schwarz lemma for $\Dee$ -- i.e. by formulating an inequality that is valid without
assuming that the holomorphic mapping in question has a fixed point. This generalisation is as 
follows:

\begin{theorem}\label{T:spec} Let $G\in\hol(\OM_n;\OM_n), \ n\geq 2$, and define $d_G:=$ the
degree of the minimal polynomial of $G(0)$. Then:
\begin{equation}\label{E:growthBound}
r(G(X)) \ \leq \ \frac{r(X)^{1/d_G}+r(G(0))}{1+r(G(0))r(X)^{1/d_G}} \quad\forall X\in\OM_n.
\end{equation}
Furthermore, the inequality \eqref{E:growthBound} is sharp in the sense that there exists a
non-empty set $\exep_n\subset\OM_n$ such that given any $A\in\exep_n$ and $d=1,\dots,n$, we can
find a $\shrp\in\hol(\OM_n;\OM_n)$ such that
\begin{align}
d_{\shrp} \ &= \ d, \;\; \text{and}\notag\\
r(\shrp(A)) \ &= \ \frac{r(A)^{1/d}+r(\shrp(0))}{1+r(\shrp(0))r(A)^{1/d}}. \label{E:sharp}
\end{align}
\end{theorem}
\medskip

\section{A Discussion of Observation~\ref{O:wontDo}}\label{S:MainObs}

We begin this discussion with a couple of definitions from complex geometry. Given a domain
$\OM\subset\Cn$, the {\em Carath{\'e}odory pseudodistance} between two points 
$z_1,z_2\in\OM$ is defined as
\[
c_\OM(z_1,z_2) \ := \ \sup\left\{\poinc(f(z_1),f(z_2)):f\in\hol(\OM;\Dee)\right\},
\]
where $\poinc$ is the Poincar{\'e} distance on $\Dee$ (and $\poinc$ is 
given by $\poinc(\zt_1,\zt_2)={\rm tanh}^{-1}(\mobi(\zt_1,\zt_2))$ for $\zt_1,\zt_2\in\Dee$).
In the same setting, the {\em Lempert functional} on $\OM\times\OM$, is defined as
\begin{multline}
\Lem{\OM}(z_1,z_2) \\ 
:= \ \inf\left\{\poinc(\zt_1,\zt_2):\exists\psi\in\hol(\Dee;\OM) \ 
\text{and $\zt_1,\zt_2\in\Dee$ such that $\psi(\zt_j)=z_j, \ j=1,2.$}\right\}.
\end{multline}
It is not hard to show that the set on the right-hand side above is non-empty. The reader
is referred to Chapter~III of \cite{jarnickiPflug:idmca93} for details. Next, we examine a few 
technical objects. For the remainder of this section, $S=(s_1,\dots,s_n)$ will denote a point in
$\Cn, \ n\geq 2$. For $z\in\cDee$ define the rational map
$f_n(z;S):=(\ess_1(z;S),\dots,\ess_{n-1}(z;S)), \ n\geq 2$, by
\begin{align}
\ess_j(z;S) \ := \ \frac{(n-j)s_j-z(j+1)s_{j+1}}{n-zs_1}, \quad S&\in\Cn \ \text{s.t.} \ 
									n-zs_1\neq 0,\notag \\
j&=1,\dots,(n-1).\notag
\end{align}
Next, define
\[
F(Z;\bcdot) \ := \ f_2(z_1;\bcdot)\circ\dots\circ f_n(z_{n-1};\bcdot) 
\quad\forall Z=(z_1,\dots,z_{n-1})\in\cDee^{n-1},
\]
where the second argument varies through that region in $\Cn$ where the right-hand side above 
is defined. The connection of these objects with our earlier discussions is established via
\[
f(z;S)\ := \ \frac{\sum_{j=1}^njs_j(-1)^jz^{j-1}}
                                {\sum_{j=0}^{n-1}(n-j)s_j(-1)^jz^{j}}, \quad z\in\cDee,
\]
and $S$ varies through that region in $\Cn$ where the right-hand side above is defined.
Note the resemblance of $f(z;S)$ to $\f(z;W)$ defined earlier. From Theorem~3.5 of
\cite{costara:osNPp05}, we excerpt:

\begin{result}\label{R:costaraKey} Let $S=(s_1,\dots,s_n)$ denote a point in $\Cn$. Then:
\begin{enumerate}
\item[1)] $f(z;\bcdot)=F(z,\dots,z;\bcdot) \ \forall z\in\cDee$, wherever defined.
\item[2)] $S\in G_n$ if and only if $\sup_{z\in\cDee}|f(z;S)|<1$, $n\geq 2$.
\item[3)] If $S\in G_n, \ n\geq 2$, then
\[
\sup_{z\in\cDee}|f(z;S)| \ = \ \sup_{Z\in\cDee^{n-1}}|F(Z;S)|.
\]
\end{enumerate}
\end{result}
\smallskip

For convenience, let us refer to the Carath{\'e}odory pseudodistance on $G_n, \ n\geq 2$, 
by $c_n$. Next, define -- here we refer to Section~2 of \cite{nikPflThoZwo:ecmsp07} -- the 
following distance function on $G_n$ 
\begin{equation}\label{E:nptzKey}
\symdist(S,T) \ := \ \max_{Z\in(\bdy)^{n-1}}\poinc(F(Z;S),F(Z;T)) \quad\forall S,T\in G_n.
\end{equation}
This is the distance function -- whose properties have been studied in 
\cite{nikPflThoZwo:ecmsp07} -- we shall exploit to support Observation~\ref{O:wontDo}.
The well-definedness of the right-hand side above follows from parts (2) and (3) of
Result~\ref{R:costaraKey} above. Furthermore, since $F(Z;S),F(Z;T)\in\Dee$ for each $Z\in\cDee^{n-1}$
whenever $S,T\in G_n, \ n\geq 2$, it follows simply from the definition that
\begin{equation}\label{E:ineq1}
c_n(S,T) \ \geq \ \symdist(S,T) \quad\forall S,T\in G_n.
\end{equation}
Since we have now adopted certain notations from \cite{nikPflThoZwo:ecmsp07}, we must make the
following
\smallskip

\noindent{{\bf Note.} {\em We have opted to rely on the notation of \cite{costara:osNPp05}. 
This leads to a slight discrepancy between our definition of $\symdist$ in \eqref{E:nptzKey}
and that in \cite{nikPflThoZwo:ecmsp07}. This discrepancy is easily reconciled by the
observation that $F(\bcdot;S)$ used here and in \cite{costara:osNPp05} will have to be
read as $F(\bcdot;-s_1,s_2,\dots,(-1)^ns_n)$ in \cite{nikPflThoZwo:ecmsp07}. This is
harmless because $S\in G_n\iff (-s_1,s_2,\dots,(-1)^ns_n)\in G_n$.}
\smallskip

Let us now refer back to the condition \eqref{E:necc} with $M=2$. An easy calculation involving
$2\times 2$ matrices reveals that
\[
\text{\em When $M=2$, \eqref{E:necc}}\iff \ \sup_{z\in\cDee}
\left|\frac{\f(z;W_1)-\f(z;W_2)}{1-\overline{\f(z};W_2)\f(z;W_1)}\right|\leq \hyper{\zt_1}{\zt_2}.
\]
If $W_1$ is nilpotent of order $n$ (recall that all matrices occuring in 
\eqref{E:necc} are non-derogatory), then $\f(\bcdot;W_1)\equiv 0$. Of course, $W_2\in\OM_n$
implies that $(s_1(W_2),\dots,s_n(W_2))\in G_n$. By part (2) of Result~\ref{R:costaraKey}, 
$\f(z;W_2)\in\Dee \ \forall z\in\cDee$. This leads to the following key fact:
\begin{multline}\label{E:ineq2}
\text{\em When $M=2$ and $W_1$ is nilpotent of order $n$}, \\
\eqref{E:necc} \ \iff \ \sup_{z\in\cDee}{\rm tanh}^{-1}|\f(z;W_2)|=\sup_{z\in\cDee}\poinc(0,\f(z;W_2))
\leq \poinc(\zt_1,\zt_2).
\end{multline}
\smallskip

We now appeal to Proposition~2 in \cite{nikPflThoZwo:ecmsp07}, i.e. 
$\symdist(0,\bcdot)\neq c_n(0,\bcdot)$ for each $n\geq 3$. Let us now fix $n\geq 3$.
Let $S_0\in G_n\setminus\{0\}$ be such that $c_n(0,S_0)>\symdist(0,S_0)$. Let $\eps_0>0$
be such that $c_n(0,S_0)=\symdist(0,S_0)+2\eps_0$. Let us write
$S_0=(s_{0,1},\dots,s_{0,n})$ and choose two matrices $W_1,W_2\in\OM_n$ as follows:
\[
W_1=\text{\em a nilpotent of order $n$,}\quad
W_2=\begin{bmatrix}
                \ 0  & {} & {} & (-1)^{n-1}s_{0,n} \ \\
                \ 1  & 0  & {} & (-1)^{n-2}s_{0,n-1} \ \\
                \ {} & \ddots  & \quad \ddots & \vdots\qquad \ \\
                \ \text{\LARGE{0}} & {}  &  1  & s_{0,1}\qquad \
                \end{bmatrix}_{n\times n},
\]
i.e. $W_2$ is the companion matrix of the polynomial $z^n+\sum_{j=1}^n(-1)^js_{0,j}z^{n-j}$.
We emphasize the following facts that follow from this choice of $W_1$ and $W_2$ 
\begin{align}
\f(\bcdot,W_1) &= f(\bcdot;0,\dots,0) \equiv 0, \; \; \f(\bcdot,W_2) = 
f(\bcdot;S_0), \label{E:equiv} \\
\text{\em $W_1$ and $W_2$} \ &\text{\em are, by construction, non-derogatory.}\notag
\end{align}
The relations in \eqref{E:equiv} are cases of a general
correspondence between matrices in $\OM_n$ and points in $G_n$ , given
by the surjective, holomorphic map $\Pi_n:\OM_n\longrightarrow G_n$, where
\[
\Pi_n(W) \ := \ (s_1(W),\dots,s_n(W)),
\]
and $s_j(W), \ j=1,\dots n$, are as defined in the beginning of this article.
\smallskip

Let us pick two distinct points $\zt_1,\zt_2\in\Dee$ such that
\begin{equation}\label{E:choice}
\poinc(\zt_1,\zt_2)-\eps_0 \ < \ \symdist(0,S_0) \ \leq \ \poinc(\zt_1,\zt_2).
\end{equation}
Assume, now, that \eqref{E:necc} is a {\em sufficient condition} for the existence
of a $\hol(\Dee;\OM_n)$-interpolant. Then, in view of the choices of $W_1,W_2$,
the second inequality in \eqref{E:choice}, and \eqref{E:equiv} we get
\begin{equation}\label{E:ineq3}
\sup_{z\in\cDee}{\rm tanh}^{-1}|\f(z;W_2)| \ = \ 
\sup_{z\in\bdy}{\rm tanh}^{-1}|\f(z;W_2)| \ = \ \symdist(0,S_0) \ \leq \ \poinc(\zt_1,\zt_2).
\end{equation}
The first equality in \eqref{E:ineq3} is a consequence of part~(2) of 
Result~\ref{R:costaraKey}: since $S_0\in G_n$, the rational function 
$\f(\bcdot;W_2)=f(\bcdot;S_0)\in\hol(\Dee)\bigcup\smoo(\cDee)$, whence the equality
follows from the Maximum Modulus Theorem. But now, owing to the equivalence
\eqref{E:ineq2}, the estimate \eqref{E:ineq3} implies, by assumption, that
there exists an interpolant $F\in\hol(\Dee;\OM_n)$ such that $F(\zt_j)=W_j, \ j=1,2.$
Then, $\Pi_n\circ F:\Dee\longrightarrow G_n$ satisfies $\Pi_n\circ F(\zt_1)=0$
and $\Pi_n\circ F(\zt_2)=S_0$. Then, by the definition of the Lempert functional
(for convenience, we denote the Lempert functional of $G_n$ by $\Lem{n}$)
\begin{align}
\Lem{n}(0,S_0) \ \leq \ \poinc(\zt_1,\zt_2) \ &< \ \symdist(0,S_0)+\eps_0 
					&&\text{(from \eqref{E:choice}, 1st part)}\notag \\
		&< \ c_n(0,S_0). &&\text{(by definition of $\eps_0$)}\notag
\end{align}
But, for any domain $\OM$, the Carath{\'e}odory pseudodistance and the Lempert function
always satisfy $c_\OM\leq\Lem{\OM}$. Hence, we have just obtained a contradiction.
Hence our assumption that \eqref{E:necc} is sufficient for the existence of an
$\hol(\Dee,\OM_n)$-interpolation, for $n\geq 3$, must be false. 
\medskip 

\section{The Proof of Theorem~\ref{T:disc}}\label{S:proofDisc}

The proofs in this section depend crucially on a theorem by Vesentini. The result is as follows:

\begin{result}[Vesentini, \cite{vesentini:spr68}]\label{R:vesentini}
Let $\banal$ be a complex, unital Banach
algebra and let $r(x)$ denote the spectral radius of any element $x\in\banal$.
Let $f\in\hol(\Dee;\banal)$. Then, the function $\zt\longmapsto r(f(\zt))$ is
subharmonic on $\Dee$.
\end{result}
\smallskip

The following result is the key lemma of this section. The proof of Theorem~\ref{T:disc} is reduced 
to a simple application of this lemma. The structure of this proof is reminiscent of 
\cite[Theorem~1.1]{nokraneRansford:Slam01}. This stems from the manner in which Vesentini's
theorem is used. The essence of the trick below goes back to Globevnik \cite{globevnik:Slsr74}. The 
reader will notice that Theorem~\ref{T:disc} specialises to Globevnik's Schwarz lemma when $W_1=0$.
\smallskip

\begin{lemma}\label{L:key} Let $F\in\hol(\Dee;\OM_n)$. For each $\lam\in\sigma(F(0))$,
define $m(\lam):=$the multiplicity of $\lam$ as a zero of the
minimal polynomial of $F(0)$. Define the Blaschke product
\[
B(\zt) \ := \ \prod_{\lam\in\sigma(F(0))}\blah{\zt}{\lam}^{m(\lam)}, \quad\zt\in D.
\]
Then $|B(\mu)|\leq|\zt| \ \forall\mu\in\sigma(F(\zt))$.
\end{lemma}
\begin{proof}
The Blaschke product $B$ induces a matrix function $\widetilde{B}$ on $\OM_n$: for any matrix
$A\in\OM_n$, we set
\[
\widetilde{B}(A) \ := \ \prod_{\lam\in\sigma(F(0))}(\id-\overline{\lam} A)^{-m(\lam)}
(A-\lam\id)^{m(\lam)},
\]
which is well-defined on $\OM_n$ because whenever $\lam\neq 0$,
\[
(\id-\overline{\lam}A) \ = \ \overline{\lam}(\id/\overline{\lam}-A)\in GL(n,\cplx).
\]
Furthermore, since $\zt\longmapsto(\zt-\lam)/(1-\overline{\lam}\zt)$, $|\lam|<1$, has a 
power-series expansion that converges uniformly on  compact subsets of $\Dee$, it follows from 
standard arguments that
\begin{equation}\label{E:specB}
\sigma(\widetilde{B}(A)) \ = \ \{B(\mu):\mu\in\sigma(A)\}\quad\text{for any $A\in\OM_n$.}
\end{equation}
By the definition of the minimal polynomial, $\widetilde{B}\circ F(0)=0$. 
Since $\widetilde{B}\circ F(0)=0$, there exists a holomorphic map 
$\Phi\in\hol(\Dee;M_n(\cplx))$ such that $\widetilde{B}\circ F(\zt)=\zt\Phi(\zt)$.
Note that
\begin{equation}\label{E:specRelation}
\sigma(\widetilde{B}\circ F(\zt)) \ = \ \sigma(\zt\Phi(\zt)) \ = \
\zt\sigma(\Phi(\zt))\quad\forall\zt\in\Dee.
\end{equation}
Since $\sigma(\widetilde{B}\circ F(\zt))\subset\Dee$, the above equations give us:
\begin{equation}\label{E:circBound}
r(\Phi(\zt)) \ < \ 1/R \quad\forall\zt:|\zt|=R, \ R\in(0,1).
\end{equation}
Taking $\banal=M_n(\cplx)$ in Vesentini's theorem, we see that
$\zt\longmapsto r(\Phi(\zt))$ is subharmonic on the unit disc. Applying the Maximum Principle to
\eqref{E:circBound} and taking limits as $R\longrightarrow 1^-$, we get
\begin{equation}\label{E:oneBound}
r(\Phi(\zt)) \ \leq \ 1\quad\forall\zt\in\Dee.
\end{equation}
In view of \eqref{E:specB}, \eqref{E:specRelation} and \eqref{E:oneBound}, we get
\[
|B(\mu)| \ \leq \ |\zt|r(\Phi(\zt)) \ \leq \ |\zt|\quad\forall\mu\in\sigma(F(\zt)).
\]
\end{proof}

We are now in a position to provide

\begin{custom}\begin{proof}[{\bf The proof of Theorem~\ref{T:disc}.}] Define the
disc automorphisms
\[
M_j(\zt) \ := \ \frac{\zt-\zt_j}{1-\overline{\zt_j}\zt}, \quad j=1,2,
\]
and write $\Phi_j=F\circ M_j^{-1}, \ j=1,2$. Note that $\Phi_1(0)=W_1$. 
For $\lam\in\sigma(W_1)$, let $m(\lam)$ be as stated in the theorem.
Define the Blaschke product
\[
B_1(\zt) \ := \ \prod_{\lam\in\sigma(W_1)}\blah{\zt}{\lam}^{m(\lam)}, \quad\zt\in\Dee.
\]
Applying Lemma~\ref{L:key}, we get
\begin{align}
\hyper{\zt_1}{\zt_2} \ = \ |M_1(\zt_2)| \ &\geq \
\prod_{\lam\in\sigma(W_1)}\hyper{\mu}{\lam}^{m(\lam)} \notag \\
&= \ \prod_{\lam\in\sigma(W_1)}\mobi(\mu,\lam)^{m(\lam)}
 \quad\forall\mu\in\sigma(\Phi_1(M_1(\zt_2)))=\sigma(W_2).\label{E:1stpart}
\end{align}
Now, swapping the roles of $\zt_1$ and $\zt_2$ and applying the same argument to
\[
B_2(\zt) \ := \ \prod_{\mu\in\sigma(W_2)}\blah{\zt}{\mu}^{m(\mu)}, \quad\zt\in\Dee,
\]
we get
\begin{equation}\label{E:2ndpart}
\hyper{\zt_1}{\zt_2} \ \geq \ \prod_{\mu\in\sigma(W_2)}\mobi(\lam,\mu)^{m(\mu)}
\quad\forall\lam\in\sigma(W_1). \end{equation}
Combining \eqref{E:1stpart} and \eqref{E:2ndpart}, we get 
\[
\max\left\{\max_{\mu\in\sigma(W_2)}\prod_{\lam\in\sigma(W_1)}\mobi(\mu,\lam)^{m(\lam)},
\ \max_{\lambda\in\sigma(W_1)}\prod_{\mu\in\sigma(W_2)}\mobi(\lam,\mu)^{m(\mu)}\right\} \
\leq \ \hyper{\zt_1}{\zt_2}.
\]
\end{proof}
\end{custom}

We conclude this section with an example.
\smallskip

\begin{example}\label{Ex:Obs2} {\em An illustration of Observation~\ref{O:compare}}
\smallskip

We begin by pointing out that the phenomenon below is expected for $n=2$. We want 
to consider $n>2$ and show that {\em there is no interpolant for the following data}, 
but that this cannot be inferred from \eqref{E:necc}. First
the matricial data: let $n=2m, \ m\geq 2$, and let
\begin{align}
W_1 \ &= \ \text{\em any block-diagonal matrix with two $m\times m$-blocks that}\notag \\
      &\qquad\text{\em are each nilpotent of order $m$.}
\end{align}
Next, for an $\alpha\in\Dee$, $\alpha\neq 0$, let
\[
W_2 \ = \ \text{\em the companion matrix of the polynomial $(z^{2m}-\alpha z^m)$.}
\]
Note that, by construction, $W_2$ is non-derogatory. We have the characteristic 
polynomials $\chi^{W_1}(z)=z^m$ and $\chi^{W_2}(z)=z^{2m}-\alpha z^m$. Hence
\[
\f(\bcdot;W_1) \ \equiv \ 0, \qquad \f(z;W_2) \ = \ \frac{-m\alpha z^{m-1}}{2m-m\alpha z^m}.
\]
We recall, from Section~\ref{S:MainObs}, the following equivalent form of \eqref{E:necc}:
\begin{equation}\label{E:another}
\text{\em When $M=2$, \eqref{E:necc}}\iff \ \sup_{z\in\cDee}
\left|\frac{\f(z;W_1)-\f(z;W_2)}{1-\overline{\f(z};W_2)\f(z;W_1)}\right|\leq \hyper{\zt_1}{\zt_2}.
\end{equation}
Since, clearly, $\f(\bcdot;W_2)\in\hol(\Dee)\bigcap\smoo(\cDee)$, by the Maximum 
Modulus Theorem 
\begin{align}
\sup_{z\in\cDee}
\left|\frac{\f(z;W_1)-\f(z;W_2)}{1-\overline{\f(z};W_2)\f(z;W_1)}\right| \ &= \
\sup_{z\in\bdy}\frac{m|\alpha|}{|2m-m\alpha z^m|} \notag \\
&= \ \frac{m|\alpha|}{2m-m|\alpha|} \ < \ |\alpha|.\label{E:data1}
\end{align}
\smallskip

Observe that $\sigma(W_1)=\{0\}$ and
$\sigma(W_2)=\{0,|\alpha|^{1/m}e^{i(2\pi j+{\sf Arg}(\alpha))/m}, \ j=1,\dots,m\}$. Therefore,
\begin{align}
\max_{\mu\in\sigma(W_2)}\prod\nolimits_{\lam\in\sigma(W_1)}\mobi(\mu,\lam)^{m(\lam)} \ &= 
\ |\alpha|, \notag \\
\max_{\lambda\in\sigma(W_1)}\prod\nolimits_{\mu\in\sigma(W_2)}\mobi(\lam,\mu)^{m(\mu)} \ &= \ 0.
\notag
\end{align}
We set $\zt_1=0$ and pick $\zt_2\in\Dee$ in such a way that
\begin{equation}\label{E:data2}
\frac{m|\alpha|}{2m-m|\alpha|} \ < \ |\zt_2| \ = \ \hyper{\zt_1}{\zt_2} \ < \ |\alpha|.
\end{equation}
Such a choice of $\zt_2$ is made possible by the inequality \eqref{E:data1}. In view
of the last calculation above, we see that the data-set $((W_1,\zt_1),(W_2,\zt_2))$
constructed violates the inequality \eqref{E:SchwarzIneq}. Thus, there is no
$\hol(\Dee,\OM_{2m})$-interpolant for this data-set. In contrast, since the 
equivalent form \eqref{E:another} of \eqref{E:necc} is satisfied, the latter
does not yield any information about the existence of a $\hol(\Dee,\OM_{2m})$-interpolant.\qed
\end{example} 
\medskip

\section{The Proof of Theorem~\ref{T:spec}}\label{S:proofSpec}
  
In order to prove Theorem~\ref{T:spec}, we shall need the following elementary
  
\begin{lemma}\label{L:mobiusTrans} Given a fractional-linear transformation 
$T(z):=(az+b)/(cz+d)$, if $T(\bdy)\Subset\cplx$, then $T(\bdy)$ is a circle with
\[
{\rm centre}(T(\bdy)) \ = \ \frac{b\overline{d}-a\overline{c}}{|d|^2-|c|^2}, \qquad
{\rm radius}(T(\bdy)) \ = \ \frac{|ad-bc|}{||d|^2-|c|^2|}.
\]
\end{lemma}
\smallskip
  
We are now in a position to present
\smallskip
  
\begin{custom}\begin{proof}[{\bf The proof of Theorem~\ref{T:spec}.}] Let $G\in\hol(\OM_n;\OM_n)$
and let $\lam_1,\dots,\lam_s$ be the distinct eigenvalues of $G(0)$. Define $m(j):=$the multiplicity
of the factor $(\lam-\lam_j)$ in the minimal polynomial of $G(0)$. Define the Blaschke product
\[
B_G(\zt) \ := \ \prod_{j=1}^s\blah{\zt}{\lam_j}^{m(j)}, \quad\zt\in\Dee.
\]
$B_G$ induces the following matrix function which, by a mild abuse of notation, we shall also denote
as $B_G$
\[
B_G(Y) \ := \ \prod_{j=1}^s(\id-\overline{\lam_j} Y)^{-m(j)}(Y-\lam_j\id)^{m(j)} \quad\forall Y\in\OM_n,
\]
which is well-defined on $\OM_n$ precisely as explained in the proof of Lemma~\ref{L:key}. Once again,
owing to the analyticity of $B_G$ on $\OM_n$,
\[
\sigma(B_G(Y)) \ = \ \{B_G(\lam):\lam\in\sigma(Y)\}\quad\forall Y\in\OM_n,
\]
whence $B_G:\OM_n\mapp\OM_n$. Therefore, if we define
\[
H(X) \ := \ B_G\circ G(X) \quad\forall X\in\OM_n,
\]
then $H\in\hol(\OM_n;\OM_n)$ and, by construction, $H(0)=0$. By the Ransford-White result,
$r(H(X))\leq r(X)$, or, more precisely
\[                      
\max_{\mu\in\sigma(G(X))}\left\{\prod_{j=1}^s\hyper{\mu}{\lam_j}^{m(j)}\right\} \ \leq \
r(X) \quad\forall X\in\OM_n.
\]
In particular:
\[
\max_{\mu\in \sigma(G(X))}\left[\mobd(\mu;\sigma(G(0)))^{d_G}\right] \ \leq \ 
r(X) \quad\forall X\in\OM_n,
\]
where, for any compact $K\varsubsetneq\Dee$ and $\mu\in\Dee$, we define
$\mobd(\mu; K):=\min_{\zt\in K}\ilnhyper{\mu}{\zt}$. For the moment, let us fix 
$X\in\OM_n$. For each $\mu\in\sigma(G(X))$, let $\lammu$ be an
eigenvalue of $G(0)$ such that
$\ilnhyper{\mu}{\lammu}=\mobd(\mu;\sigma(G(0)))$. Now fix $\mu\in\sigma(G(X))$. The
above inequality leads to
\begin{equation}\label{E:hyperIneq1}
\hyper{\mu}{\lammu} \ \leq \ r(X)^{1/d_G}.
\end{equation}
Applying Lemma~\ref{L:mobiusTrans} to the M{\"o}bius transformation
\[
T(z) \ = \ \frac{|\mu|z-\lammu}{1-\overline{\lammu}|\mu|z},
\]
we deduce that
\[
\hyper{\zt}{\lammu} \ \geq \ \frac{||\mu|-|\lammu||}{1-|\mu||\lammu|}\quad\forall\zt:|\zt|=|\mu|.
\]
Applying the above fact to \eqref{E:hyperIneq1}, we get
\begin{align}
\frac{|\mu|-|\lammu|}{1-|\mu||\lammu|} \ &\leq \ r(X)^{1/d_G} \notag \\
\Rightarrow\quad |\mu| \ &\leq \frac{r(X)^{1/d_G}+|\lammu|}{1+|\lammu|r(X)^{1/d_G}},
\quad\mu\in\sigma(G(X)).
\label{E:specIneq}
\end{align}
Note that the function
\[
t\longmapsto\frac{r(X)^{1/d_G}+t}{1+r(X)^{1/d_G}t}, \quad t\geq 0,
\]
is an increasing function on $[0,\infty)$. Combining this fact with \eqref{E:specIneq}, we get
\[
|\mu| \ \leq \ \frac{r(X)^{1/d_G}+r(G(0))}{1+r(G(0))r(X)^{1/d_G}},
\]
which holds $\forall\mu\in\sigma(G(X))$, while the right-hand side is independent of $\mu$. Since
this is true for any arbitrary $X\in\OM_n$, we conclude that
\[
r(G(X)) \ \leq \ \frac{r(X)^{1/d_G}+r(G(0))}{1+r(G(0))r(X)^{1/d_G}} \quad\forall X\in\OM_n.
\]
\smallskip

In order to prove the sharpness of \eqref{E:SchwarzIneq}, let us fix an $n\geq 2$, and define
\[
\exep_n \ := \ \{A\in\OM_n:A \ \text{has a single eigenvalue of multiplicity $n$} \}.
\]
Pick any $d=1,\dots,n$, and define
\[
M_d(X) \ := \ \begin{cases}
                \quad [\tr(X)/n],        & \text{if $d=1$}, \\
                \ \begin{bmatrix}
                        \ 0  & {} & {} & \tr(X)/n \ \\
                        \ 1  & 0  & {} & 0 \ \\
                        \ {} & \ddots & \ddots & \vdots \ \\
                        \ {} & {} & 1 & 0 \
                        \end{bmatrix}_{d\times d}, & \text{if $d\geq 2$},
                \end{cases}
\]
and, for the chosen $d$, define $\Shrp$ by the following block-diagonal matrix
\[
\Shrp(Y) \ := \begin{bmatrix}
                        \ M_d(X) & {} \ \\
                        \ {} & \dfrac{\tr(X)}{n}\mathbb{I}_{n-d} \
                        \end{bmatrix} \quad\forall X\in\OM_n.
\]
For our purposes $\shrp=\Shrp$ for each $A\in\exep_n$; i.e., the equality \eqref{E:sharp} will
will hold with the same function for each $A\in\exep_n$. To see this, note that
\begin{itemize}
\item $r(\Shrp(X))=|\tr(X)/n|^{1/d}$; and
\item $\Shrp(0)$ is nilpotent of degree $d$, whence $d_{\Shrp}=d$.
\end{itemize}
Therefore,
\[
\frac{r(A)^{1/d}+r(\Shrp(0))}{1+r(\Shrp(0))r(A)^{1/d}} \ = \
r(A)^{1/d} \ = \ r(\Shrp(A)) \quad\forall A\in\exep_n,
\]
which establishes \eqref{E:sharp}
\end{proof}
\end{custom}

\end{document}